\title{Integral action feedback design for conservative abstract systems in the presence of input nonlinearities}

\author{Ling Ma, Vincent Andrieu, Daniele Astolfi, 
Mathieu Bajodek, Cheng-Zhong Xu and Xuyang Lou
\thanks{
L. Ma and X. Lou are with School of IoT Engineering, Jiangnan University, Wuxi 214122, China
({\tt \small  lingma@stu.jiangnan.edu.cn, Louxy@jiangnan.edu.cn}).}
\thanks{V. Andrieu, D. Astolfi, M. Bajodek and C. Xu are with 
Universit\'e Claude Bernard Lyon 1, CNRS, LAGEPP UMR 5007, 43 boulevard du 11 novembre 1918, F-69100, Villeurbanne, France 
({\tt \small name.surname@univ-lyon1.fr}).
}
\thanks{This work was supported by the China Scholarship Council (No. 202206790086) and 
by the French Grant ODISSE ANR-19-CE48-0004-01 and
ALLIGATOR ANR-22-CE48-0009-01.}
}

\documentclass[journal,twoside,web,9pt]{ieeecolor}

\usepackage{generic}
\usepackage{graphicx}
\usepackage{epstopdf}
\usepackage{graphics,psfrag,color,theorem}
\usepackage{amsfonts,latexsym}
\usepackage[noadjust]{cite}

\usepackage[utf8]{inputenc}
\usepackage{tabularx,ragged2e,booktabs}

\usepackage{mathrsfs}

\usepackage{amssymb,amsmath,mathtools}
\usepackage{enumerate}
\usepackage{float}
\graphicspath{{Figures/}}

\usepackage{hyperref}
\usepackage{bigints}

\usepackage{varwidth}
\usepackage{comment}
\usepackage{subcaption}
\usepackage{soul}

\usepackage[noadjust]{cite}

\newcommand{\cA}{\mathcal A}
\newcommand{\cB}{\mathcal B}
\newcommand{\cC}{\mathcal C}

\newcommand{\cF}{\mathcal F}
\newcommand{\cG}{\mathcal G}

\newcommand{\cI}{\mathcal I}

\newcommand{\cK}{\mathcal K}
\newcommand{\cL}{\mathcal L}
\newcommand{\cM}{\mathcal M}
\newcommand{\cN}{\mathcal N}

\newcommand{\cP}{\mathcal P}

\newcommand{\cX}{\mathcal X}

\DeclareMathOperator{\1}{\mathbf{1}}

\DeclareMathOperator{\Id}{I}

\theoremstyle{plain}
\newtheorem{theorem}{Theorem}

\newtheorem{proposition}{Proposition}
\newtheorem{assumption}{Assumption}
\newtheorem{lemma}{Lemma}

\newtheorem{definition}{Definition}

\newtheorem{assumption*}{Assumption}

\def\RR{{\mathbb R}}    

\def\RR{{\mathbb R}}    

\def\cX{\mathcal{X}}

\def\cK{\mathcal{K}}

\def\cI{\mathcal{I}} 

\def\LL{\mathcal {L}}

\newcommand{\ddt}{{\tfrac{{\rm d}}{{\rm d}t}}}

\def\dd{{\rm d}\hbox{\hskip 0.5pt}}

\begin{document}

\maketitle

\begin{abstract}
In this article, we present a stabilization feedback law  with integral action for conservative abstract linear systems subjected to actuator nonlinearity. Based on the designed control law, we first prove the well-posedness and global asymptotic stability of the origin of the closed-loop system by constructing a weak Lyapunov functional. Secondly, as an illustration, we apply the results to a wave equation coupled with an ordinary differential equation (ODE) at the boundary. Finally, we give the simulation results to illustrate the effectiveness of our method. 
\end{abstract}


\maketitle

\section{Introduction}
In recent years, there has been a growing emphasis on studying the stabilization problem of infinite-dimensional systems subjected to nonlinearities in control inputs. Such nonlinearities impose limitations on actuators which is common in various physical systems and can lead to undesirable outcomes if overlooked.
Among the large literature on this topic, 
we recall, 
in the context of stabilization, 
the series of works
\cite{prieur2016wave, marx2017cone, jacob2020remarks, chitour2021one,marx2021forwarding}.
 In \cite{prieur2016wave}  and \cite{jacob2020remarks}, the application of the Lyapunov theory is the primary analytical approach for addressing the stability aspects within the respective problems. In \cite{chitour2021one}, the authors transform the wave equation into a discrete-time dynamical system, which allows for the analysis of the stabilization. Additionally, both \cite{marx2017cone} and \cite{marx2021forwarding} tackle the stabilization problem of infinite-dimensional systems by leveraging the dissipativity property.
The problem of input nonlinearities 
becomes however even more critical  when an 
integral action is added in the loop.

The use of integral action 
provide robustness
tracking properties and is well established 
in the context of infinite-dimensional systems, 
see, e.g. 
\cite{pohjolainen1982robust,xu1995robust,dos2008boundary,paunonen2015controller,Terrand2020,coron2019pi,vanspranghe2023output,balogoun2023iss,astolfi2022global,lorenzetti2023saturating}.
The feedback design is typically based 
on a first preliminary feedback 
ensuring stability of the controlled system, 
and then a second feedback employing a 
integrator gain which has to be chosen 
small enough. In this context, a lot of effort
has been put in characterizing 
a bound for the regulating gain. 
Such a characterization has been 
done by employing transfer function
approaches, e.g. 
\cite{xu1995robust,logemann1998integral,logemann1999integral,paunonen2015controller}, 
 forwarding-Lyapunov based approaches, 
see, e.g. \cite{terrand2019regulation,astolfi2022global,balogoun2023iss}, and more recently
singular perturbations
\cite{lorenzetti2023saturating}.

In this article, 
we study the use of integral action 
feedback for abstract linear systems
subject to input nonlinearities.
Existing methodologies 
(see, e.g. \cite{logemann1998integral,logemann1999integral,lorenzetti2023saturating}) 
rely 
on the property that before adding the integral gain, the controlled system needs
to be already open-loop stable or stabilized
via a preliminary feedback.
However, as explained in \cite{astolfi2022global}, 
the use of preliminary-feedback 
is unfeasible in the case of bounded nonlinearity.
The goal of this paper is therefore
to relax such an open-loop stability 
assumption and to consider a wider class of 
systems for which integral-action controller
can be employed. 
To this end, we focus on abstract-linear
systems in which the open-loop dynamics
is only conservative, but not necessarily asymptotically stable, and 
we propose a new feedback controller
solving the aforementioned problem.
The solution is mainly inspired by the work 
\cite{slemrod1989feedback}, which involves the use of  weak Lyapunov functionals, and the 
forwarding approach recently extended
to infinite-dimensional systems 
(see, e.g. \cite{Terrand2020,marx2021forwarding,vanspranghe2023output,lorenzetti2023saturating}).

The rest of this article is organized as follows.
In Sections~\ref{sec:problem}-\ref{sec:proofs}, we examine the scenario of an abstract system and provide the general sufficient conditions that enable us to introduce a control law solving the stabilization of the origin problem. One crucial condition is that the control operator
is bounded.
However, in Section~\ref{sec:general}, we illustrate through an example that the general approach developed in this article allows  to consider systems with boundary control (i.e. 
unbounded control operators). This is demonstrated using the example of a wave equation.

\emph{Notations.} $\RR$ denotes the set of real numbers and $\RR_+=[0,\infty)$. $|\cdot|$ stands for the Euclidean norm, and $\Id$ represents the identity matrix. 
Given a Hilbert space $\cX$, we denote with $\cI_\cX$  the identity operator. Let $H^1(0, 1) \subset L^2(0, 1)$ be the Hilbert space of real-valued, absolutely continuous functions defined over the interval $[0, 1]$, with square-integrable derivatives. Given two Hilbert space $\cX_1$ and $\cX_2$, $\cL(\cX_1,\cX_2)$ represents the class of linear bounded operators from $\cX_1$ to $\cX_2$, and $\cL(\cX_1) = \cL(\cX_1,\cX_1)$. Given an operator $\cA$ on a Hilbert space $\cX$, 
$D(\cA)$ denotes its domain, 
$\cA^\ast$  its adjoint operator and $\rho(\cA)$  its resolvent set.

\section{Problem statement and preliminaries}
\label{sec:problem}
We define a Hilbert space $\cX$ equipped with a scalar product $\langle \cdot, \cdot \rangle_\cX$ and the induced norm $\|\cdot \|_\cX$. We focus on the following abstract control framework:
\begin{equation}\label{abstractform}
\left\{
\begin{aligned}
\ddt x &=\cA x  + \cB \psi(u(t)), \\
\ddt z &=\cC x,  \\
x(0) &=x_0, z(0)=z_0,
\end{aligned}
\right.
\end{equation} 
where $\cA: D(\cA) \subset \cX \to \cX$
is a linear operator which is an infinitesimal generator of a strongly continuous semigroup denoted 
$(e^{t\cA})_{t\geq 0}$, $\cB:\RR \to \cX$
is a linear bounded operator, $\cC: D(\cA) \to \RR$ is a linear $\cA$-bounded and admissible operator for $\cA$, i.e., the operator $\cC$ is $\cA$-admissible, $x_0 \in \cX$, $z_0 \in \RR$, 
and $\psi:\RR\to\RR$ is a function satisfying the following definition.

\begin{definition}\label{def:generalized_sat}
The function $\psi:\RR\to\RR$ is a Lipschitz monotonic nonlinearity if it satisfies the following properties
\begin{enumerate}
\item\label{delip1} $\psi(s)= 0$ $\Leftrightarrow$ $s=0$;
\item\label{delip2} $\exists L>0$, $|\psi(s_1)- \psi(s_2)|\leq L |s_1- s_2|$, $\forall(s_1,s_2)\in \RR^2$;
\item\label{delip3} $(\psi(s_1)-\psi(s_2))(s_1-s_2)\geq 0$, $\forall (s_1,s_2)\in \RR^2$.
\end{enumerate}
\end{definition}

The main objective of this paper is to construct a feedback law with the aim of ensuring the well-posedness of the system \eqref{abstractform} and attaining global asymptotic stability for the zero solution. It is noteworthy that the system can be conceptualized as a cascade system, encompassing both an infinite-dimensional component, possibly represented by a partial differential equation (PDE), and an ordinary differential equation (ODE).

A typical scenario in which such a system structure appears is often the case of a closed-loop system with a control law that includes integral action and saturation. The saturation case, for instance, has been investigated in \cite{astolfi2022global}. However, here we do not assume that the operator $\cA$ generates an exponentially stable semi-group. Thus, when the control input is zero, the $x$ component of the state doesn't necessarily go to zero. Additionally, since we do not necessarily assume that the dynamics in $z$ are controlled, we do not consider the anti-windup method developed in the same article. Exploring the incorporation of anti-windup action could be a subject for future studies.

To begin, we make the following set of assumptions.

\begin{assumption}\label{ass_dissipativity}
The linear operator $\cA: D(\cA) \subset \cX \rightarrow \cX$, with domain  $D(\cA)$ dense in $\cX$, generates a $C_0$-semigroup denoted by $(e^{\cA t})_{t \geq 0}$ and satisfies the following:
\begin{enumerate}
\item \label{weakness} there exists a coercive self-adjoint positive operator $\cP \in \LL(\cX)$ such that
\begin{equation}\label{disseq}
\langle \cA x, \cP x
\rangle_\cX 
+
\langle x, \cP \cA x
\rangle_\cX \leq 0;
\end{equation}
\item \label{inverofa} it is an invertible operator with $\cA^{-1}$ belonging to $\LL (\cX, D(\cA))$;
\item \label{compactofa} $D(\cA)$ is compactly injected in $\cX$.
\end{enumerate}
\end{assumption}

We remark even in the case in which the function 
$\psi$ is linear, invertibility 
of the operator $\cA$ is a necessary condition 
for the stabilizability of the system 
\eqref{abstractform}, as well established
in the theory of output regulation, see, e.g. 
\cite{pohjolainen1982robust,paunonen2015controller}.

In the following, we use the state space $\cX_e:=\cX \times \RR$ equipped with the norm $\|(x,z) \|_{\cX_e}=\|x\|_\cX+|z|$.
Denoting the state by $\xi:=\begin{pmatrix} x \\ z \end{pmatrix}$, the system described earlier in equation  \eqref{abstractform} can be expressed in a more compact form as follows:  
\begin{equation} \label{compactform}
\dot \xi = \cF \xi +\cG  \psi(u(t)), \qquad
\xi(0) = \begin{pmatrix}x_0\\z_0\end{pmatrix},
\end{equation}
where the operator $\cF:D(\cF) \to \cX_e$ and $\cG$ in $\cL(\cX_e)$ are defined as 
$$\cF:=\begin{pmatrix} \cA & 0 \\ \cC & 0 \end{pmatrix}, 
\quad  \cG:=\begin{pmatrix} \cB \\ 0 \end{pmatrix}, 
$$
and
$ D(\cF)= D(\cA)\times\RR$.
According to \cite[Lemma 1]{xu1995robust}, we can conclude that the operator $\cF$ generates a $C_0$-semigroup denoted by $(e^{t \cF})_{t\geq0}$.

\section{Controller Design} 
\label{sec:controller}

\subsection{State feedback design} 

First of all, let $\cM: \cX \rightarrow \RR$ in $\cL(\cX,\RR)$ be defined by
\begin{equation}\label{definitionm}
\cM  x = \cC \cA^{-1} x,\quad  \forall x \in D(\cA).
\end{equation}
Note that by Assumption \ref{ass_dissipativity}, the operator $\cA^{-1}$ is in $\cL(\cX, D(\cA))$. 
Hence $\cM$ is well-defined and since $\cC$ is $\cA$-admissible, it yields that  $\cM$ belongs to  $\cL(\cX,\RR)$.
Note also that
\begin{equation}\label{eq_forwarding}
\cM \cA x = \cC x,\quad \forall x \in D(\cA).
\end{equation}
With the operator $\cM$ and a positive real number $\mu>0$, the state feedback controller is selected as follows:
\begin{equation}\label{controller}
   u =  \cK \xi,\qquad \cK=\begin{pmatrix}-\cB^\ast \cP- \mu \cB^\ast \cM^\ast \cM & & \mu \cB^\ast \cM^\ast\end{pmatrix},
\end{equation}
where $\cP$ is defined in~\eqref{disseq}, and $\xi$ is given by the compact form in \eqref{compactform}.
As a consequence, the closed-loop system reads as follows: 
\begin{equation}\label{closedloop}
\ddt \xi = \cF_{\psi} \xi, \qquad
\xi(0)=\begin{pmatrix}x_0\\z_0\end{pmatrix},
\end{equation} 
where 
$$
\cF_{\psi} \xi=\begin{pmatrix} \cA x +\cB \psi(-\cB^\ast\left(\cP x- \mu \cM^\ast (z- \cM x)\right)) \\ \cC x \end{pmatrix},$$
with the domain of operator $\cF_\psi$ given as
\begin{equation}\label{domainofdasig}
D(\cF_\psi):=D(\cF)=D(\cA)\times \RR . 
\end{equation}

\subsection{Well-posedness and stability}

With the previously described control law and within the framework of Assumption \ref{ass_dissipativity}, it is possible to attain an initial result stating that the closed-loop system is well-posed in the sense that a concept of solution exists. Additionally, it is demonstrated that the zero solution is globally stable in the Lyapunov sense.

\begin{theorem}[Well-posedness and stability]\label{theoremofpo}
Let $\psi$ be a Lipschitz bounded monotonic nonlinearity (see Definition \ref{def:generalized_sat}).
For any $\xi_0 \in \cX_e$ (resp. $\xi_0 \in D(\cF_{\psi})$), there exists a unique mild (resp. classical) solution denoted $e^{\cF_\psi t}\xi_0 \in C^0(\RR_{+};\cX_e)$ to \eqref{closedloop} (resp. $C^1(\RR_{+};\cX_e) \bigcap C^0(\RR_{+};D(\cF_{\psi}))$). 
Moreover, the origin is stable. In other words, there exists $k>0$ such that for each 
$\xi_0 \in \cX_e$, 
\begin{equation}\label{eq_stab}
\|e^{\cF_\psi t}\xi_0\|_{\cX_e} \leq k \|\xi_0\|_{\cX_e}, \forall t\geq 0.
\end{equation}
\end{theorem}
The proof of this theorem is provided in Section \ref{sec_ProofTheo1}. Its crucial element is to demonstrate that the operator $\cF_\psi$ is an $m$-dissipative operator for a specific Hilbert space, whose norm is equivalent to the norm of the space $\cX_e$. Once this is established, the result on well-posedness and stability follows by invoking the theory of nonlinear semigroups on Hilbert space \cite[Theorem 4.20]{Isao1992} or \cite[Theorem 3.1, page 54]{brezis1973ope}.

Note that even in the particular case in which $\psi$ is linear, this theorem does not appear to be written in the literature. Comparing it with the article \cite{Terrand2020} and other works, there is no assumption that the operator $\cA$ is exponentially stable. Consequently, the control law appears more intricate. For instance, it does not seem generally feasible to employ a proportional control of the form $u=-k_i z$ with a suitably chosen coefficient $k_i$. 
Note that item~\ref{compactofa}) in Assumption \ref{ass_dissipativity} is an assumption that can be removed if $\psi$ is linear.
This one was not introduced in \cite{Terrand2020} or \cite{astolfi2022global,lorenzetti2023saturating}.
Indeed, such an assumption is needed to allow the use of the Schauder fixed-point theorem and ensures the so-called \textit{range condition} for $m$-dissipative operators. Note that such an assumption will be also invoked in the next theorem claiming asymptotic stability of the origin when applying LaSalle arguments.

We can observe that an assumption of coercivity is imposed on the operator $\cP$. 
Once again, this assumption stems from the fact that, in demonstrating stability, we establish the $m$-dissipative nature of the closed-loop operator. This characteristic is also pivotal in proving the subsequent theorem. It remains an open question whether such a result is attainable without relying on $m$-dissipativity (by employing Datko-type arguments or strong Lyapunov functionals).

Also, it is not established that this control law guarantees asymptotic convergence of trajectories toward the origin. As we will see in the following section, additional assumptions need to be introduced to ensure that the solutions converge to the system's origin.

\subsection{Asymptotic convergence result}

In the proof of the preceding result, it was demonstrated that the operator $\cF_\psi$ is $m$-dissipative. The natural question arises: is the dissipated quantity sufficient to guarantee convergence towards the origin of the system, thereby ensuring global asymptotic stability of this equilibrium point? Drawing inspiration from the findings in \cite{slemrod1989feedback}, we are aware that observability is a key property that can facilitate such a property. The observability assumptions we propose for the system are as follows.
\begin{assumption}\label{ass_Obs}
The following two conditions hold.
\begin{enumerate}
    \item \label{condiofin} The condition $\cC \cA^{-1} \cB \neq 0$ is satisfied. 
\item  \label{condiofobs} The pair $(\cA,\cB^*\cP)$ is
approximately observable. In other words, 
$$
\cB^*\cP e^{\cA t} x_0 = 0, \forall t\geq 0\  \Rightarrow x_0 = 0.
$$
\end{enumerate}
\end{assumption}

The first condition is already introduced in Pohjolainen's seminal paper \cite{pohjolainen1982robust} concerning the addition of integral action. It is worth noting that, in finite dimensions, this condition is a necessary requirement for stabilization when the function $\psi$ is linear.
The second condition is introduced to handle the control nonlinearity.

\begin{theorem}[Asymptotic stability]\label{theoremofsta}
Let  Assumptions \ref{ass_dissipativity}, \ref{ass_Obs} hold.
Then, the origin of the closed-loop system \eqref{closedloop} with the control law \eqref{controller} is globally 
asymptotically stable. In other words, \eqref{eq_stab} is satisfied and for each $\xi_0$ in $\cX_e$
and moreover
\begin{equation*}
\lim_{t\rightarrow +\infty}\|e^{\cF_\psi t}\xi_0\|_{\cX_e}= 0.
\end{equation*}
\end{theorem}

The proof of this theorem is postponed in Section \ref{sec_ProofAssymptotic}. This one is decomposed as follows. In the first step, it is shown that the output $\cK \xi$ is an approximately observable output for the semi-groups $t\mapsto e^{\cF t}$. In the second step, it is shown that we can apply a LaSalle procedure.

\section{Proofs}
\label{sec:proofs}
\subsection{Proof of Theorem \ref{theoremofpo}}
\label{sec_ProofTheo1}
In order to prove this result, we first establish that the operator $\cF_\psi$ is a maximal dissipative operator with respect to a particular norm.

\begin{lemma}\label{lemma_Coercivity}
Let the inner product $\langle \cdot, \cdot \rangle_{V}$ be 
defined as
\begin{align}\label{innerproductinv}
\langle \begin{psmallmatrix}x_1\\z_1\end{psmallmatrix}, \begin{psmallmatrix}x_2\\z_2\end{psmallmatrix}\rangle_{V} &= \langle \cP x_1, x_2\rangle_\cX \notag\\
&~~~+ \mu (z_1-\cM x_1) (z_2-\cM x_2), 
\end{align}
where $\cP$ is defined in Point \ref{weakness}) of Assumption \ref{ass_dissipativity}, $\mu$ is a positive scalar and $\cM$ is defined in \eqref{definitionm}.
The norm $\| \cdot \|_{V}$ induced by this inner product is equivalent to the usual norm $\| \cdot \|_{\cX_e} $ in $\cX_e$, i.e., there exists $\underline{v}>0$ and $\overline{v}>0$ such that
\begin{equation}
\underline{v} \|\xi\|_{\cX_e}\leq \|\xi\|_V\leq   \bar v\|\xi\|_{\cX_e}  .
\end{equation}
\end{lemma}

\begin{proof}
By using Cauchy-Schwarz inequality, we have 
\begin{align}\label{minv}
\!\! \! \| \begin{psmallmatrix}x\\z\end{psmallmatrix} \|_{V} \! \leq \! \left( \| \cP \|_{\cL(\cX)} \! + \! 2\mu \|\cM \|^2_{\cL(\cX, \RR)} \right)\|x\|_{\cX}^2 \!+\! 2\mu |z|^2
\end{align}
for any $(x, z) \in \cX_e$.
Due to the coercive property of $\cP$, we have the following inequality for a positive constant $c > 0$:
\begin{equation}\label{inequa1}
\langle \cP x, x \rangle_\cX \geq c \| x \|^2_{\cX}.
\end{equation}
Additionally, we have the following inequality for every $\epsilon \in (0,1)$ and $w_1, w_2 \in \RR$:
\begin{equation}\label{inequ2}
|w_1 - w_2|^2 \geq \epsilon \left( \frac{1}{2} |w_1|^2 -|w_2|^2 \right).
\end{equation}
Combining these two inequalities \eqref{inequa1} and \eqref{inequ2}, we obtain the following:
\begin{align} \label{maxv}
\| \begin{psmallmatrix}x\\z\end{psmallmatrix} \|_{V} &\geq c \|x\|^2_{\cX} + \epsilon \mu \left(\frac{1}{2}|z|^2 - \|\cM \|^2_{\cL(\cX, \RR)} \| x\|^2_{\cX}\right) \notag \\
&\geq \left( c - \epsilon \mu \|\cM \|^2_{\cL(\cX, \RR)} \right) \|x\|^2_{\cX} +  \frac{\epsilon \mu}{2}|z|^2.
\end{align}
With \eqref{minv} and \eqref{maxv}, we can choose $\epsilon$ sufficiently small to have a positive lower bound. Thus, we demonstrate the equivalence of the norms $\| \cdot \|_{\cX_e} $ and $\| \cdot \|_{V}$. 

\end{proof}

\begin{proposition}\label{Prop_MDiss}
The operator $\cF_\psi$ is $m$-dissipative with respect to the  scalar product $\langle\cdot,\cdot\rangle_V$, i.e.,
for each $(\xi_1,\xi_2)$  in $D(\cF_\psi)^2$, it yields
\begin{equation}
\langle \cF_\psi(\xi_1) - \cF_\psi(\xi_2),\xi_1-\xi_2 \rangle_V \leq 0,
\end{equation}
and moreover the range condition is satisfied, i.e., for each $\lambda>0$ and $\xi \in \cX_e$, there exists  $\tilde \xi\in D(\cF_\psi)$ such that 
\begin{equation}
    (\lambda \cI_{\cX_e} - \cF_\psi)\tilde\xi = \xi.
\end{equation}
\end{proposition}
\begin{proof}
For the first part of the statement, 
let {\color{blue}$\xi_1=(x_1, z_1)$ and $\xi_2=(x_2, z_2)$} be in $D(\cF_\psi)$. We introduce $\tilde x = x_1-x_2$, $\tilde z = z_1 -z_2$, and $u_i=\cK \xi_i$ with $i\in\{1,2\}$ and $\cK$ in~\eqref{controller}. Using the inner product defined by \eqref{innerproductinv}, we can compute $\langle \cF_\psi (\xi_1)- \cF_\psi(\xi_2), \xi_1 - \xi_2 \rangle_V$ following these steps:
\begin{align*}
& \langle  \cF_\psi (\xi_1)- \cF_\psi(\xi_2), \xi_1 - \xi_2 \rangle_V  \\
& \qquad \quad 
\begin{aligned}
&=\langle \cP \cA \tilde x + \cP \cB\left(\psi(u_1)-\psi(u_2)\right), \tilde x \rangle_\cX  \\
&~~~+ \mu \left( \cC \tilde x - \cM(\cA \tilde x + \cB(\psi(u_1) - \psi(u_2))\right) (\tilde z - \cM \tilde x)  \\
&=\langle \cP \cA \tilde x, \tilde x \rangle_\cX + \langle \left(\psi(u_1)-\psi(u_2)\right), \cB^\ast \cP \tilde x \rangle_\RR  \\
&~~~-\mu \left( \cM \cB(\psi(u_1) - \psi(u_2)) \right) (\tilde z - \cM \tilde x) . 
\end{aligned}
\end{align*}
With the first point of Assumption \ref{ass_dissipativity}, it implies
\begin{align*}
 &\langle  \cF_\psi (\xi_1) - \cF_\psi(\xi_2), \xi_1 - \xi_2 \rangle_V   \\
& \qquad \quad
\begin{aligned}
&\leq (\psi(u_1) -\psi(u_2)) \cB^\ast \cP \tilde x  \\
&~~~-\mu (\psi(u_1) -\psi(u_2)) \cB^\ast \cM^\ast (\tilde z - \cM \tilde x)  \\
& = -(\psi(u_1) -\psi(u_2)) (-\cB^\ast \cP \tilde x +\mu \cB^\ast \cM^\ast (\tilde z - \cM \tilde x) ) \notag \\
&= - (\psi(u_1) -\psi(u_2)) (u_1-u_2) \leq 0. 
\end{aligned}
\end{align*}
For the range condition, 
consider $(x, z) \in \cX_e$, and for any positive constant $\lambda$, we need to find $(\tilde{x}, \tilde{z}) \in D(\cF_\psi)=D(\cA)\times\RR$, such that the following equations hold:
\begin{equation}\label{exisanduni}
\!\!\!\left\{
\begin{aligned}
&\lambda \tilde x - \cA \tilde x -\cB \psi \left( -\cB^\ast(\cP \tilde x -\mu \cM^\ast (\tilde z - \cM \tilde x) ) \right) =x, \\
&\lambda \tilde z -\cC \tilde x =z.
\end{aligned}
\right.
\end{equation} 
Since $\cA$ is assumed to be dissipative and invertible, as stated in Assumption \ref{ass_dissipativity}, it follows that the real parts of the eigenvalues of $\cA$ are non-negative, so any positive constant $\lambda$ belongs to the resolvent set $\rho(\cA)$. Then the operator $(\lambda \cI_{\cX} -\cA)^{-1}$ exists. From the first line, one has 
\begin{multline}\label{tildex}
\tilde x = (\lambda \cI_{\cX} -\cA)^{-1} \\\times\Big[[ \cB \psi\left( \cB^\ast (-\cP \tilde x + \mu \cM^\ast (\tilde z -\cM \tilde x))\right) +x \Big].
\end{multline}
From the second line of \eqref{exisanduni}, one has 
\begin{equation}\label{tildez}
\tilde z = \lambda^{-1}(\cC \tilde x + z). 
\end{equation}
Substituting \eqref{tildez} into \eqref{tildex}, we obtain the following expression for $\tilde x$: 
\begin{equation}\label{solutiontildex}
\tilde x = (\lambda \cI_{\cX} -\cA)^{-1} \Big[ \cB \psi(\phi(\tilde x))+x\Big],
\end{equation}
where 
\begin{equation}
\phi(\tilde x) = \cB^\ast \Big(-\cP \tilde x + \mu \cM^\ast (\lambda^{-1} (z -\cM (\lambda \cI_{\cX} - \cA)\tilde x))\Big),
\end{equation}
which implies that the first line of \eqref{exisanduni} can be expressed as an equation that depends only on $\tilde x$ given the provided data $x$ and $z$. To establish the existence of a solution to the above equation \eqref{solutiontildex}, we intend to utilize a fixed-point strategy. 
For this purpose, we define 
\begin{equation}
\cN: 
\left\{
\begin{aligned}&D(\cA)  \rightarrow  D(\cA)\subset \cX, \notag \\ 
&\bar x \mapsto \cN(\bar x),
\end{aligned}
\right.
\end{equation}
where $\cN$ is defined by
\begin{equation}\label{defN}
\cN(\bar x)\!=\! 
(\lambda \cI_{\cX} -\cA)^{-1} \Big[{\cB} \psi(\phi(\bar x))+x\Big].
\end{equation}
Once we prove the operator $\cN$ admits at least one fixed point, we have proved there exists a solution to \eqref{solutiontildex}. 
We consider the closed ball $\mathcal B(x_0;r) = \{x_0\in D(\cA): \|\cA x_0\|_\cX\leq r\}$. In the following, we need to prove that $\cN(\mathcal B(x_0;r))$ is included in a compact subset of $\cX$.
Then the Schauder fixed-point theorem \cite[Theorem B.17]{Coron2007} can be applied. 
Note that for each $\bar x\in\mathcal{B}(x_0;r)$, we have
\begin{multline}
\|\cN(\bar x)\|_{D(\cA)} \leq 
    \|(\lambda \cI_{\cX} -\cA)^{-1} \cB \|_{D(\cA)}|\psi(\phi(\bar x))| \\ + \|(\lambda \cI_{\cX} -\cA)^{-1} x \|_{D(\cA)}.
\end{multline}
But, with Definition \eqref{def:generalized_sat}, we have for each $\bar x\in\mathcal{B}(x_0;r)$
$$
|\psi(\phi(\bar x))|\leq L|\phi(\bar x)|,
$$
and we have
\begin{multline}
|\phi(\bar x)| \leq     \|\cB^\ast\|_{\mathcal L(\RR, \cX)}\|\cP\|_{\mathcal L(\cX)}\| \bar x \|_\cX \\ +
 \frac{\mu}{\lambda} \|\cB^\ast\|_{\mathcal L(\RR, \cX)} \|\cM^\ast\|_{\mathcal L(\RR, \cX)}
 \\
 \times\Big(
 |z| + \|\cM\|_{\mathcal L(\cX, \RR)}(\lambda \|\bar x\|_\cX + \|\cA \bar x\|_\cX) \Big).
\end{multline}
However, since we have for all $\bar x$ in $D(\cA)$
$$
 \|\bar x\|_\cX = \|\cA^{-1}\cA \bar x \|_\cX \leq \|\cA^{-1}\|_{\LL (\cX, D(\cA))}\|\cA \bar x \|_\cX .
$$
Consequently for each $\bar x\in\mathcal{B}(x_0;r)$
\begin{multline}
|\phi(\bar x)| \leq     \|\cB^\ast\|_{\mathcal L(\RR, \cX)}\|\cP\|_{\mathcal L(\cX)}\|\cA^{-1}\|_{\LL (\cX, D(\cA))} r \\ +
 \frac{\mu}{\lambda} \|\cB^\ast\|_{\mathcal L(\RR, \cX)} \|\cM^\ast\|_{\mathcal L(\RR, \cX)}\\
 \times \Big(
 |z| + \|\cM\|_{\mathcal L(\cX, \RR)}(\lambda r \|\cA^{-1}\|_{\LL (\cX, D(\cA))} + r) \Big). 
\end{multline}
Hence, for each $\bar x\in\mathcal{B}(x_0;r)$, 
$\|\cN(\bar x)\|_{D(\cA)}\leq \bar \cN$, where $\bar \cN$ is a positive real number.
Considering that $D(\cA)$ is compactly injected in $\cX$ {\color{blue}(as indicated in Item \ref{compactofa}) of Assumption \ref{ass_dissipativity})}, we can use the Schauder fixed-point theorem \cite[Theorem B.17]{Coron2007} to conclude that the operator $\cN$ admits a fixed point $\cN(\bar x) = \bar x$. 
In conclusion,  operator $\cF_\psi$ is maximal dissipative.
\end{proof}

\medskip
\noindent\textbf{Proof of Theorem \ref{theoremofpo}}.
With Proposition \ref{Prop_MDiss} we obtain that $\cF_\psi$ is $m$-dissipative.
From \cite[Theorem 4.20]{Isao1992}, the operator $\cF_\psi$ generates a strongly continuous semigroup of contractions on $\cX_e$ denoted by $(e^{\cF_\psi t})_{t \geq 0}$. 
With Lemma \ref{lemma_Coercivity}, it implies for each $\xi_0$ in $\cX_e$ for all $t\geq0$
\begin{equation}
\|e^{\cF_\psi t}\xi_0\|_{\cX_e} \leq \frac{1}{\underline v}\|e^{\cF_\psi t}\xi_0\|_{V}\leq \frac{1}{\underline v}\|\xi_0\|_{V}\leq \frac{\bar v}{\underline v}\|\xi_0\|_{\cX_e},
\end{equation}
and the result follows.\hfill$\blacksquare$

\subsection{Proof of Theorem \ref{theoremofsta}}
\label{sec_ProofAssymptotic}

In order to prove Theorem \ref{theoremofsta}, let us first introduce the following lemma which states that for an invertible operator with approximately observable output, the origin is the only solution which makes the output constant in time.
\begin{lemma} \label{lemmaobser}
Assume Assumptions \ref{ass_dissipativity} and Point \ref{condiofobs}) in \ref{ass_Obs} hold. Let $x_0$ be in $D(\cA)$, if there exists $y_0$ in $\RR$ such that
\begin{equation}\label{Cst_Output}
\cB^* \cP e^{\mathcal{A}t}x_0=y_0, \qquad \forall t\geq 0.
\end{equation}
Then $x_0=0$ and $y_0=0$. 
\end{lemma}

\begin{proof}
Assume that there exists  $s\in \RR_{+}$, such that 
{\color{blue}\begin{equation}\label{assotherx}
x_{1}=e^{\cA s} x_0 \neq x_0. 
\end{equation}
}
Note that, for all $t\geq0$,
\begin{equation*}\label{y1t}
\begin{aligned}
\cB^* \cP e^{\cA t} x_1
&=\cB^* \cP e^{\cA (t+s)} x_0
=y_0 
=\cB^* \cP e^{\cA t} x_0,
\end{aligned}
\end{equation*}
where the last two identities follow from \eqref{Cst_Output}.
Hence, for all $t\geq 0$, it yields
$
\cB^* \cP e^{\cA t}(x_1-x_0)=0,
$
with item~\ref{condiofobs}) in Assumption \ref{ass_Obs}, it yields $x_1=x_0$, which is a contradiction to {\color{blue}\eqref{assotherx}}.
Hence, for all $s \in \RR_{+}$, $e^{\cA s} x_0=x_0$ holds. 
This gives
$
(e^{\cA s}-I)x_0 = 0.
$
Therefore, if $x_0$ is in the domain of $\cA$,
$$
\lim_{s\to 0}\frac{(e^{\cA s}-I)x_0}{s} = \cA x_0= 0,
$$
which by the invertibility of $\cA$ as stated in Assumption \ref{ass_dissipativity}, this implies that $x_0=0$ and consequently $y_0=0$.
\end{proof}

Employing this lemma, we can now show another lemma which establishes that the operator $\cK$ involved in the definition of the control law \eqref{controller} gives an approximately observable output for the linear semi-group obtained when setting the input at zero.
\begin{lemma}\label{Lem_Obs}
Let Assumptions \ref{ass_dissipativity}, \ref{ass_Obs} hold. 
Then, for any $\mu>0$, 
the pair $(\cF, \cK)$ is approximately observable, where $\cF$ and $\cK$ are given by~\eqref{compactform} and~\eqref{controller}, respectively.
\end{lemma}

\begin{proof}
Let $\xi_0=\begin{pmatrix} x_0 \\ z_0 \end{pmatrix}$ be in $D(\cF)$ 
and assume that
\begin{align}\label{outputform}
\cK e^{ \cF  t} \xi_0 = 0 , \ \forall t\geq 0.
\end{align}
We need to show that $\xi_0 = 0$. 
Note that since $\cC$ is admissible for all $t\geq 0$, we can write
\begin{align}\label{speform}
e^{\cF t} \xi_0 = \begin{bmatrix}
    e^{\cA t} x_0\\
    z_0 + \int_0^t \cC e^{\cA s} x_0 \dd s
\end{bmatrix}.
\end{align}

According to equation \eqref{speform}, it gives that
\begin{multline*}
    \cK e^{\cF t} \xi_0 = \left( -\cB^\ast\cP-\mu\cB^\ast \cM^\ast \cM \right) e^{\cA t} x_0  \\
+\mu \cB^\ast \cM^\ast \left( z_0 + \int_0^t \cC e^{\cA s} x_0 \dd s \right).
\end{multline*}
By reorganizing the term, it yields the expression
\begin{multline*}
\cK e^{\cF t} \xi_0 = \mu \cB^\ast \cM^\ast z_0 
-\cB^\ast \cP e^{\cA t} x_0- \mu \cB^\ast \cM^\ast \varphi(t,x_0),
\end{multline*}
where
$$
\varphi(t,x_0) = \left( \cM e^{\cA t} - \int_0^t \cC e^{\cA s}\dd s \right) x_0.
$$
Note that $x_0$ is in $D(\cA)$ hence $t\mapsto \varphi(t,x_0)$ is a $C^1(\mathbb{R};\mathbb{R})$. Also, for all $t$ in $\RR_+$, we have
\begin{align*}
\ddt \varphi(t,x_0)&= \cC \cA^{-1} \cA e^{\cA t}x_0 - \cC e^{\cA t}x_0  =0.
\end{align*}
Consequently, $\varphi$ is a constant function of time and
$$
\varphi(t,x_0) = \varphi(0,x_0) = \cM x_0.
$$
It yields that 
\begin{multline*}
\cK e^{\cF t} \xi_0 = -\cB^\ast\cP  e^{{\cA}t}x_0 -\mu \cB^\ast \cM^\ast \cM  x_0+ \mu \cB^\ast \cM^\ast z_0 =0,
\end{multline*}
or,
\begin{align}\label{eqeft}
    \cB^\ast\cP  e^{{\cA}t}x_0 &= -\mu \cB^\ast \cM^\ast \cM  x_0+\mu \cB^\ast \cM^\ast z_0
    = \rm(constant).
\end{align}
Based on Lemma \ref{lemmaobser}, we get $x_0=0$ and  $-\mu \cB^\ast \cM^\ast \cM  x_0+\mu \cB^\ast \cM^\ast z_0=0$. 
Then, with Point \ref{condiofin}) in Assumption \ref{ass_Obs} $\cB^\ast \cM^\ast = (\cC \cA^{-1} \cB)^\ast\neq 0$, the remaining term $\mu \cB^\ast \cM^\ast z_0=0$ in~\eqref{eqeft} ensures $z_0=0$.
In summary, \eqref{outputform} implies $\xi_0=0$ and we can conclude that the pair $(\cF,\cK)$ is approximately observable.
\end{proof}

\medskip
\noindent\textbf{Proof of Theorem \ref{theoremofsta}}.
Even if the proof is not written exactly in \cite{slemrod1989feedback}, we can follow similar steps used in the proof of \cite[Theorem 7.1]{slemrod1989feedback} to obtain the 
desired result.
With Proposition \ref{Prop_MDiss}, $\cF_\psi$ is $m$-dissipative with respect to the norm $\langle\cdot, \cdot\rangle_V$.
As in the proof \cite[Theorem 7.1]{slemrod1989feedback}, since $\psi$ is Lipschitz (see Definition \ref{def:generalized_sat} Point \ref{delip2}))and the operator $\cA$ has a compact resolvent (Assumption \ref{ass_dissipativity} ), it is possible to show that the operator $(\lambda \cF_\psi + \Id)^{-1}$ is compact for all $\lambda>0$.
Moreover, since we have
\begin{equation}\label{eq_dotV}
\langle \xi,\cF_\psi \xi\rangle_V \leq - \psi(\cK \xi)\cK \xi\ \leq 0,  
\end{equation}
 for each $\xi_0$ in $D(\cF_\psi)$, \cite[Theorem 6.1]{slemrod1989feedback} implies that there exist $R\geq 0$ and a compact subset $\Omega$ of $D(\cF_\psi)\cap\{\xi,\langle \xi,\cF_\psi \xi\rangle_V =R\}$ such that $t\mapsto e^{\cF_{\psi}t}\xi_0$ approaches $\Omega$ as $t\rightarrow +\infty$ and
$\Omega$ is positively invariant for the flow $e^{\cF_{\psi}t}$. 
Hence, \eqref{eq_dotV} implies that for all $\tilde \xi_0$
 in $\Omega$, $\psi(\cK e^{\cF_{\psi}t}\tilde \xi_0)\cK e^{\cF_{\psi}t}\tilde \xi_0=0$ for all $t\geq 0$.
 However, with item~\ref{delip1}) in Definition \ref{def:generalized_sat}, it implies that $\cK e^{\cF_{\psi}t}\tilde \xi_0=0$, for all $t\geq 0$.
 In this case, it implies that
 \begin{align}
 \ddt e^{\cF_{\psi}t}\tilde \xi_0 &= \cF e^{\cF_{\psi}t}\tilde \xi_0  + \cG\psi(\cK e^{\cF_{\psi}t}\tilde \xi_0), \notag\\
 &=\cF e^{\cF_{\psi}t}\tilde \xi_0.
 \end{align}
 Solution generated by $\cF$ being unique, it implies
 $$
 e^{\cF_{\psi}t}\tilde \xi_0 = e^{\cF t}\tilde \xi_0, \forall t\geq 0,
 $$
 and so,
 $$
 \cK e^{\cF t}\tilde \xi_0=0, \forall t\geq 0.
 $$
 Since by Lemma \ref{Lem_Obs}, the pair $(\cF,\cK)$ is approximately observable, it implies $\tilde \xi_0=0$.
 Hence, $\Omega=0$. Consequently solution initiated from $D(\cF_\psi)$ converges to zero. Following the proof of \cite[Theorem 7.1]{slemrod1989feedback}, since $D(\cF_\psi)$ is dense in $\cX_e$, the same convergence results holds for weak solutions initialized outside $D(\cF_\psi)$.
\hfill$\blacksquare$

\section{General method and example of wave equation coupled with an ODE at the boundary}
\label{sec:general}
\subsection{The  general approach}
The preceding section demonstrated the feasibility of designing a control law for a coupled system in which the control operator (i.e., $\cB$ in \eqref{abstractform}) is bounded. In this section, we illustrate that the same control strategy can be applied in the case of boundary control, i.e. 
when the control operator is unbounded. Thus, we demonstrate, through an example of a wave equation, that the method is effective in a broader class of system then the one covered by Theorem~\ref{theoremofsta}.
A more precise theorem could be written following
\cite{tucsnak2014well} but to ease the reading 
we propose here only a general methodology highlighting 
the main steps to be followed.

The approach that has been developed to address the control objective can be summarized as follows:
\begin{description}
    \item [\textbf{Step 1:}]~~Show dissipativity of the distributed parameter part of the open loop system via a coercive Lyapunov functional, as shown in Assumption \ref{ass_dissipativity}.
    \item [\textbf{Step 2:}]~~Construction of the operator $\cM$ such that along (smooth) solution to
    \eqref{abstractform} the time function $M(t)= \cM x(t)$ satisfies
     $$
    \dot M(t) = \dot z(t)   +  m\psi(u(t)),
     $$
     for some real number $m\neq 0$. Note that in the abstract formulation $m$ is simply the real number $\cC \cA^{-1} \cB$ and we recognize equation \eqref{eq_forwarding} and the first item of Assumption \ref{ass_Obs}.
    \item [\textbf{Step 3:}]~~Construction of the control law depending on the Lyapunov functional of Step 1 and $M$ from Step 2.
    \item [\textbf{Step 4:}]~~Show $m$-dissipativity for the closed loop system.
    \item [\textbf{Step 5:}]~~Show the observability condition and convergence.
\end{description}

\subsection{The system description}

Similar to what is done in \cite{prieur2016wave}, consider the following wave equation with nonlinear control action at the boundary:
\begin{equation}\label{waveequation}
\left\{
\begin{aligned}
y_{tt}(x,t) &= y_{xx}(x,t),  &&(t,x) \in \RR_+ \times [0,1], \\
y(0,t)& =0, && t \in \RR_+, \\
y_{x}(1,t)&=\psi(u(t)), && t \in \RR_+,\\
y(x,0)&=y_0(x), y_t(x, 0) = y_{t0}(x), && x\in [0,1]
\end{aligned}
\right.
\end{equation}
where $(y_0,y_{t0})$ is the initial condition which belongs to the Hilbert space $\cX= H^1_0(0,1)\times L^2(0,1)$, and  
$\psi$ is a monotonic Lipschitz nonlinearity 
satisfying Definition \ref{def:generalized_sat}.
It has been shown in \cite{prieur2016wave} that it is possible to design a state feedback that ensures asymptotic stabilization of the zero solution. In our context, we add an integral action to the former equation:
\begin{equation}\label{eq_IntAction}
\left\{
\begin{aligned}&\dot z(t) =y(1,t), \qquad   t \in \RR_+,\\
&z(0)=z_0,
\end{aligned}\right.
\end{equation}
for some initial condition $z_0$ in $\RR$.
We aim to design a stabilizing control law for the coupled system \eqref{waveequation}-\eqref{eq_IntAction} in the state space $\cX_e = \cX\times\RR$.

It can be noticed that when $u(t)=0$ for all $t$, there exist solutions $t\mapsto y(\cdot,t)$  to the distributed parameter system \eqref{waveequation} which are not converging to zero. Note for instance that
\begin{align*}
y(x,t) &= \sin\left(\frac{\pi}{2}x\right)\cos\left(\frac{\pi}{2}t\right),\\
y_t(x,t) &= -\frac{\pi}{2}\sin\left(\frac{\pi}{2}x\right)\sin\left(\frac{\pi}{2}t\right),
\end{align*}
is a smooth solution to the $y$ part of the equation for some particular $(y_0, y_{t0})$
which doesn't converge to zero. 
This shows that the system cannot be decomposed as an exponentially stable part and an integrator.
For these reasons, the approach presented in \cite{astolfi2022global} cannot be employed.

\subsection{A dissipative system (Step 1)}
It is well known that the wave equation has a dissipativity property. This has been employed for instance in \cite{prieur2016wave} to design saturating control laws.
For instance, if one considers the following energy function,
\begin{equation}
E(t) = \frac{1}{2}\int_0^1 y_t^2(x,t)  + y_x^2(x,t)\dd x,
\end{equation}
 is well defined and $C^1$  as long as $t\mapsto \{x\mapsto (y(x,t),y_t(x,t))\}$ belongs   to $C^1(\RR_+;\cX)$.
 Along a $C^1$ solution to \eqref{waveequation}  its time derivative satisfies
\begin{equation}
\ddt E(t) = \int_0^1 y_t(x,t)y_{xx}(x,t)  + y_x(x,t)y_{xt}(x,t)\dd x.
\end{equation}
Hence with an integration by parts, 
\begin{align*}
\ddt E(t) &= [y_x(x,t)y_{t}(x,t)]_0^1  \\
&= \psi(u(t))y_t(1,t) -y_x(0,t)y_{t}(0,t).
\end{align*}
Since $y(0,t)=0$ for all $t$, it implies $y_{t}(0,t)=0$ and
\begin{equation}
\ddt E(t)=\psi(u(t))y_t(1,t).
\end{equation}
Consequently when $u(t)=0$ for all $t$, the time function $t\mapsto E(t)$ is constant along $C^1$ solution.
Since moreover, one can find the positive constants $\bar e > \underline e >0$ such that for all $t$, 
\begin{multline*}
\underline e\left(\|y(\cdot,t)\|_{H^1_0(0,1)} + \|y(\cdot,t)\|_{L^2(0,1)}\right)
\\\leq  E(t) \leq \\
\bar e\left(\|y(\cdot,t)\|_{H^1_0(0,1)} + \|y(\cdot,t)\|_{L^2(0,1)}\right) .
\end{multline*}
We get the well known dissipativity property for the wave equation.

\subsection{Construction of the control law (Steps 2 and 3)}
Let $t\mapsto \xi(t)=(y(\cdot,t),y_t(\cdot,t),z(t))$ be a solution to \eqref{waveequation}-\eqref{eq_IntAction} which belongs to $C^1(\RR_+;\cX_e)$ for a particular $t\mapsto u(t)$.
Following what was done in the former section, we consider the time function $M:\RR_+\mapsto\RR$ defined as
$$
M(t) = -\int_0^1  x y_t(x,t)\dd x, 
$$
which is well defined and $C^1$.
Its time derivative satisfies
\begin{align*}
\ddt M(t)  
&= -\int_0^1 x y_{xx}(x,t)\dd x \\
&= \int_0^1  y_{x}(x,t)\dd x - [xy_x(x,t)]_0^1 \\
&= y(1,t) - y(0,t)-\psi(u(t)) \\
&=  \dot z(t)-\psi(u(t)).
\end{align*}
Hence, a candidate Lyapunov functional is simply
\begin{equation}
V(t) = E(t) + \frac{\mu}{2}(z(t)-M(t))^2 ,
\end{equation}
with a positive scalar $\mu$.
Indeed, this one satisfies
\begin{equation}
\underline v\|\xi(t)\|_{\cX_e}\leq V(t)\leq \bar v\|\xi(t)\|_{\cX_e},
\end{equation}
for some positive real numbers $\underline v$ and $\bar v$ which don't depend on the considered solution $t\mapsto \xi(t)$. Moreover, 
\begin{align*}
\ddt V(t) 
&= \psi(u(t))\left[y_t(1,t)+\mu(z(t)-M(t))\right],
\end{align*}
and a candidate boundary control law is then
\begin{equation}\label{eq_ControlWave}
u(t) = -y_t(1,t) - \mu z(t) - \mu\int_0^1 x y_t(x,t) \dd x,
\end{equation}
since
\begin{equation}\label{eq_Lyap}
\dot V(t) \leq -u(t)\psi(u(t))\leq 0\,.
\end{equation}

\subsection{Convergence result (Steps 4 and 5)}
It is now possible to show that with the boundary controller  \eqref{eq_ControlWave}, global asymptotic stability is obtained.
Note that system \eqref{waveequation}-\eqref{eq_IntAction}-\eqref{eq_ControlWave} can be written in the compact form
\begin{equation}\label{eq_WaveOperator}
    \dot \xi = \cF_\psi \xi ,
\end{equation}
where $\cF_\psi:D(\cF_\psi){\color{blue}\rightarrow }\cX_e$ is defined by
$
    \cF_\psi \!\left(\!\begin{bmatrix}
        a\\b\\c
    \end{bmatrix}\!\right)\! \!=\! \begin{bmatrix}
        b\\a_{xx}\\a(1)
    \end{bmatrix},
$
with the (nonlinear) domain
\begin{multline}\label{eq_DomWaveCL}
    D(\cF_\psi)=\Bigg\{H^2(0,1)\times H^1(0,1)\times \RR:\\ a(0)=0,\ b(0)=0,\\  a_x(1)=\psi\left(-b(1)- \mu c- \mu \int_0^1xb(x)\dd x\right)\Bigg\}.
\end{multline}

\begin{theorem}\label{specwave}
    The origin of the system \eqref{waveequation}-\eqref{eq_IntAction} with the control law \eqref{eq_ControlWave} (or system \eqref{eq_WaveOperator}-\eqref{eq_DomWaveCL}) is globally asymptotically stable in the Hilbert space $\cX_e$.
    More precisely, for each initial condition $\xi_0=(y_0,y_{t0},z_0)$ which belongs to $D(\cF_\psi)$, there exist a unique (strong) solution $t\mapsto \xi(t)=(y(\cdot,t),y_t(\cdot,t),z(t))$ in $C^1(\RR_+;\cX_e)$ and a positive real number $k$ such that
    \begin{equation}
    \|\xi(t)\|_{\cX_e}\leq k\|\xi_0\|_{\cX_e}, 
    \qquad 
\lim_{t\to +\infty} \|\xi(t)\|_{\cX_e}=0.
    \end{equation}
\end{theorem}
\begin{proof}
\noindent \textbf{$\cF_\psi$ is dissipative in a particular Hilbert space}.
Consider the inner product~:
\begin{multline}
\langle \xi_1,\xi_2\rangle_V = \int_0^1 a_{1x}(x)a_{2x}(x) + b_1(x)b_2(x) \dd x + \\
\mu\left(c_1+\int_0^1 xb_1(x)\dd x\right)
\left(c_2+\int_0^1 xb_2(x)\dd x\right).
\end{multline}
It can be checked that for all $(\xi_1,\xi_2)$ in $D(\cF_\psi)^2$
\begin{multline}
\langle \xi_1-\xi_2,\cF_\psi\xi_1-\cF_\psi\xi_2\rangle_V = \int_0^1 \tilde a_{x}(x)\tilde b_{x}(x) + \tilde b(x)\tilde a_{xx}(x) \dd x + \\
 \mu  \left(\tilde c+\int_0^1 x\tilde b(x)\dd x\right)
\left(\tilde a(1)+\int_0^1 x\tilde a_{xx}(x)\dd x\right),
\end{multline}
where $\tilde a = a_1-a_2$, $\tilde b=b_1-b_2$ and $\tilde c = c_1-c_2$.
With an integration by parts, it yields,
\begin{multline}
\langle \xi_1-\xi_2,\cF_\psi\xi_1-\cF_\psi\xi_2\rangle_V = [\tilde a_x(x)\tilde b(x)]_0^1 + \\
\mu \left(\tilde c+\int_0^1 x\tilde b(x)\dd x\right)
\left(\tilde a(1)-\tilde a(1)+\tilde a(0) +[x\tilde a_{x}(x)]_0^1\right).
\end{multline}
Hence,
\begin{multline}
\langle \xi_1-\xi_2,\cF_\psi\xi_1-\cF_\psi\xi_2\rangle_V \\= \tilde a_x(1)\left[\tilde b(1)+\mu \tilde c+\mu \int_0^1x\tilde b(x)\dd x\right] ,
\end{multline}
which gives,
\begin{multline}
\langle \xi_1-\xi_2,\cF_\psi\xi_1-\cF_\psi\xi_2\rangle_V \\= (\psi(u_1)-\psi(u_2))(u_2-u_1)\leq 0 .
\end{multline}
Hence $\cF_\psi$ is a dissipative operator in Hilbert space induced by the scalar product $V$.

\noindent \textbf{About the range condition.}
Let $(\bar a, \bar b, \bar c)$ be in $\cX_e= H^1_0(0,1)\times L^2(0,1)\times\RR $. We are looking for a triplet $(a,b,c)$ in $D(\cF_\psi)$ such that
$$
\cF_\psi \!\left(\!\begin{bmatrix}
        a\\b\\c
    \end{bmatrix}\!\right)\! - \lambda \begin{bmatrix}a\\b\\c \end{bmatrix}= \begin{bmatrix}
    \bar a\\\bar b\\\bar c
\end{bmatrix} ,
$$
for all $\lambda> 0$.
More precisely, we are looking for $(a,b,c)$ in $H^2(0,1)\times H^1(0,1)\times \RR$ solution to the following equations
\begin{align}\label{eq_cond1}
        b-\lambda a=\bar a, \ a_{xx}-\lambda b=\bar b, \ 
        a(1) -\lambda c =\bar c, \\ \label{eq_cond2}
        a(0)=0 , \ b(0)=0, \
        a_x(1)=\psi(\mathfrak u),
\end{align}
where 
$$
\mathfrak u=-b(1)-\mu c-\mu\int_0^1xb(x)\dd x.
$$
This gives
\begin{equation}\label{EDO_a}
a_{xx} - \lambda^2 a =\bar b + \lambda \bar a, 
\end{equation}
and
\begin{multline}
\mathfrak u=-\frac{\mu+\lambda^2 }{\lambda}a(1)-\mu \lambda \int_0^1xa(x)\dd x\\ +\frac{\mu }{\lambda}\bar c-\bar a(1)-\mu \int_0^1x\bar a(x)\dd x.
\end{multline}
By integrating \eqref{EDO_a} with $a(0)=0$ in \eqref{eq_cond2}, it yields
\begin{multline}\label{eq_a}
\begin{bmatrix}
a(x)\\
a_x(x)
\end{bmatrix}  = \exp(Sx)\begin{bmatrix}0\\a_x(0)\end{bmatrix}
\\+\int_0^x \exp(S(x-s))\begin{bmatrix}
    0\\
    \bar b(s)+\lambda \bar a(s)
\end{bmatrix}\dd s ,
\end{multline}
where
$$
S = \begin{bmatrix}
    0&1\\
    \lambda^2& 0
\end{bmatrix}.
$$
Consequently,
$$
\mathfrak u = \alpha_0 a_x(0) + \beta_0 , \ a_x(1) = \alpha_1 a_x(0) + \beta_1,
$$
for some real numbers $(\beta_0,\beta_1)$, which depends on $(\lambda, \bar a,\bar b, \bar c)$ and
where
\begin{align*}
\alpha_0 &= \begin{bmatrix}1&0\end{bmatrix}
\left[-\frac{\mu+\lambda^2}{\lambda} \exp(S)
- \mu \lambda \int_0^1 x \exp(Sx)\dd x\right]\begin{bmatrix}0\\1\end{bmatrix},\\
\alpha_1 &= \begin{bmatrix}0&1\end{bmatrix}\exp(S)\begin{bmatrix}0\\1\end{bmatrix}.
\end{align*}
With an integration by parts, it yields
\begin{multline*}
\alpha_0 =  \begin{bmatrix}1&0\end{bmatrix}
\Big[\left(-\frac{\mu +\lambda^2}{\lambda} I + {\lambda\mu }{S^{-2}} -{\lambda\mu}{S^{-1}} \right) \exp(S)  \\
- {\lambda \mu}{S^{-2}}\Big]\begin{bmatrix}0\\1\end{bmatrix}.
\end{multline*}
Note that since $S^2=\lambda^2 \Id$, it yields
\begin{equation}
\alpha_0 =  \begin{bmatrix}1&0\end{bmatrix}
\left[\left(-\lambda I -{\lambda  \mu} {S^{-1}}\right) \exp(S)
- \frac{ \mu}{ \lambda }\right]\begin{bmatrix}0\\1\end{bmatrix}.
\end{equation}
Note that it can be checked that $\alpha_0<0$ and $\alpha_1>0$ for any $\lambda>0$. 
The function $\psi$ being monotonic and odd, it yields 
\begin{align*}
\lim_{r\to+\infty} \alpha_1 r +\beta_1 - \psi(\alpha_0 r + \beta_0)&=+\infty,\\
\lim_{r\to-\infty} \alpha_1 r +\beta_1 - \psi(\alpha_0 r + \beta_0)&=-\infty.
\end{align*}
Hence, for any $(\bar a, \bar b, \bar c)$  and $\lambda>0$, there always exists $a_x(0)$ such that $a_x(1) = \psi(\mathfrak u)$. Hence, the function $a$ obtained from \eqref{eq_a}, and $(b,c)=(\bar a + \lambda a,\frac{a(1)-\bar c}{\lambda})$ satisfy \eqref{eq_cond1} and \eqref{eq_cond2} and the triple $(a, b, c)$ belongs to $D(\cF_\psi)$. Consequently, the range condition holds.

Since the operator is $m$-dissipative, with \cite[Theorem 3.1, page 54]{brezis1973ope} there exists a semi-group of contraction $t\mapsto e^{\cF_\psi t}$ such that for each $\xi$ in $D(\cF_\psi)$, $\xi(t)=e^{\cF_\psi t}\xi$ is a solution to \eqref{eq_WaveOperator} which belongs to $C^1(\RR_+;\cX_e)$ and such that
\begin{equation}\label{eq_CV}
\|e^{\cF_\psi t}\xi\|_V \leq \|\xi\|_V,\qquad 
  \|\cF_\psi e^{\cF_\psi t}\xi\|_V \leq \|\xi\|_V.
\end{equation}

\noindent \textbf{About the observability condition.}
The next step of the procedure is to look at the observability condition.
Since $V$ is a Lyapunov functional which satisfies equation \eqref{eq_Lyap},
we need to show that  $e^{\cF_\psi t}\xi=0$ is the only solution which satisfies
\begin{equation}
0=u(t)=-y_t(1,t)- \mu z(t)- \mu \int_0^1 x y_t(x,t)\dd x, \ \forall t\geq 0.
\end{equation}
But looking at the time derivative of the former expression, this implies
$$
 z^{(3)}(t)= y_{tt}(1,t)=0. 
$$
But since with \eqref{eq_CV}, $z(t)$ is bounded, it implies $y_t(1,t)=0$ for all $t\geq 0$.

So we are looking for the solution to
\begin{equation*}
\left\{
\begin{aligned}
y_{tt}(x,t)  &= y_{xx}(x,t), && (t,x) \in \RR_+ \times [0,1], \\
y(0,t)& =0,  && t \in \RR_+, \\
y_{x}(1,t)&=0, && t \in \RR_+,\\
y_t(1,t)&=0,  && t \in \RR_+,\\
y(x,0)&=y_0(x), y_t(x, 0) = y_{t0}(x), &&x\in [0,1],
\end{aligned}
\right.
\end{equation*}
which we know is only $0$. This implies that $z(t)=0$ for all $t\geq0$.

\noindent \textbf{Convergence.}
The right hand side of \eqref{eq_CV} implies  
the existence of $\mathfrak{c}>0$ such that
$$
\|y_{t}(\cdot,t)\|_{L^2(0,1)} + \|y_{xx}(\cdot,t)\|_{L^2(0,1)}+|z(t)|<\mathfrak{c}.
$$
This implies precompactness of the trajectory and allows one to conclude global asymptotic stability employing standard argument (see \cite{slemrod1989feedback}).
\end{proof}

\subsection{Simulation results}
\begin{figure}
  \centering
  \begin{subfigure}{4cm}
    \includegraphics[width=4cm]{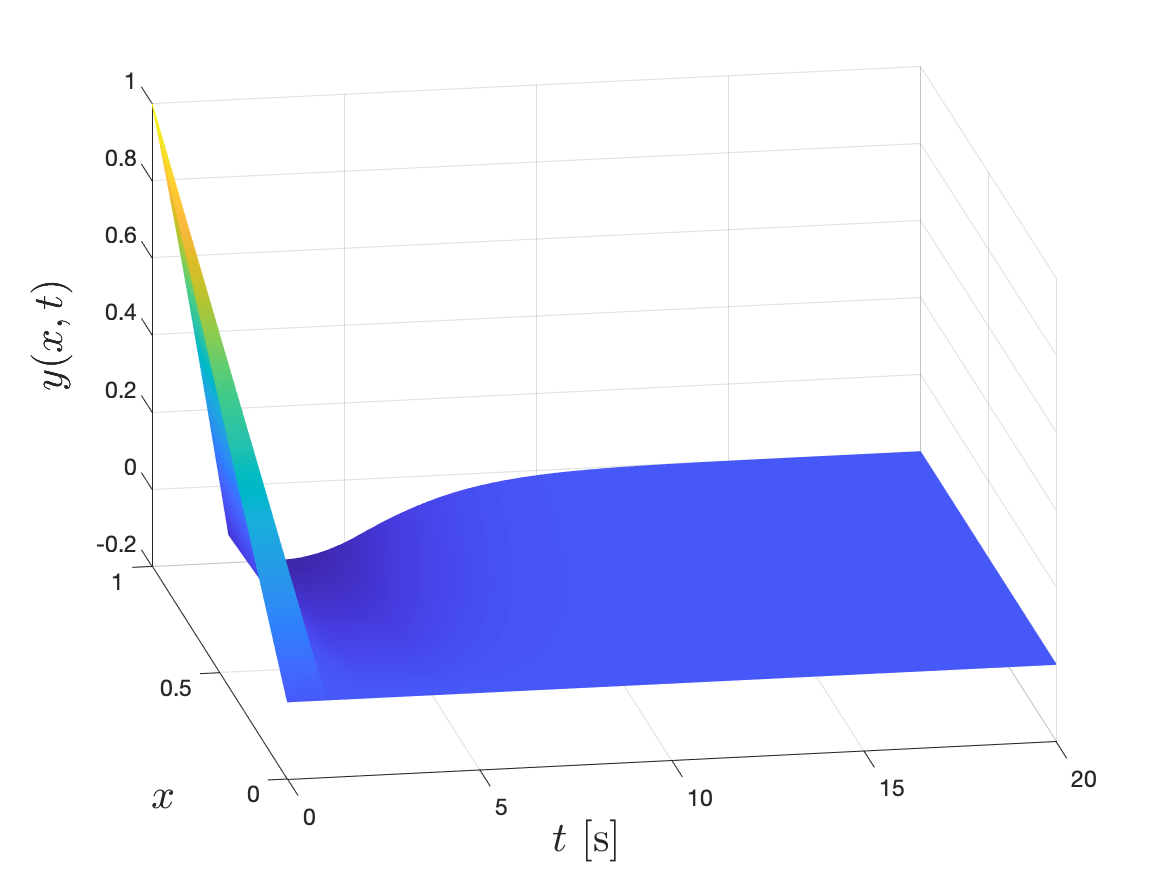}
    \caption{Solution $y(x,t)$.}\label{yxt}
  \end{subfigure}
   \begin{subfigure}{4cm}
     \includegraphics[width=4cm]{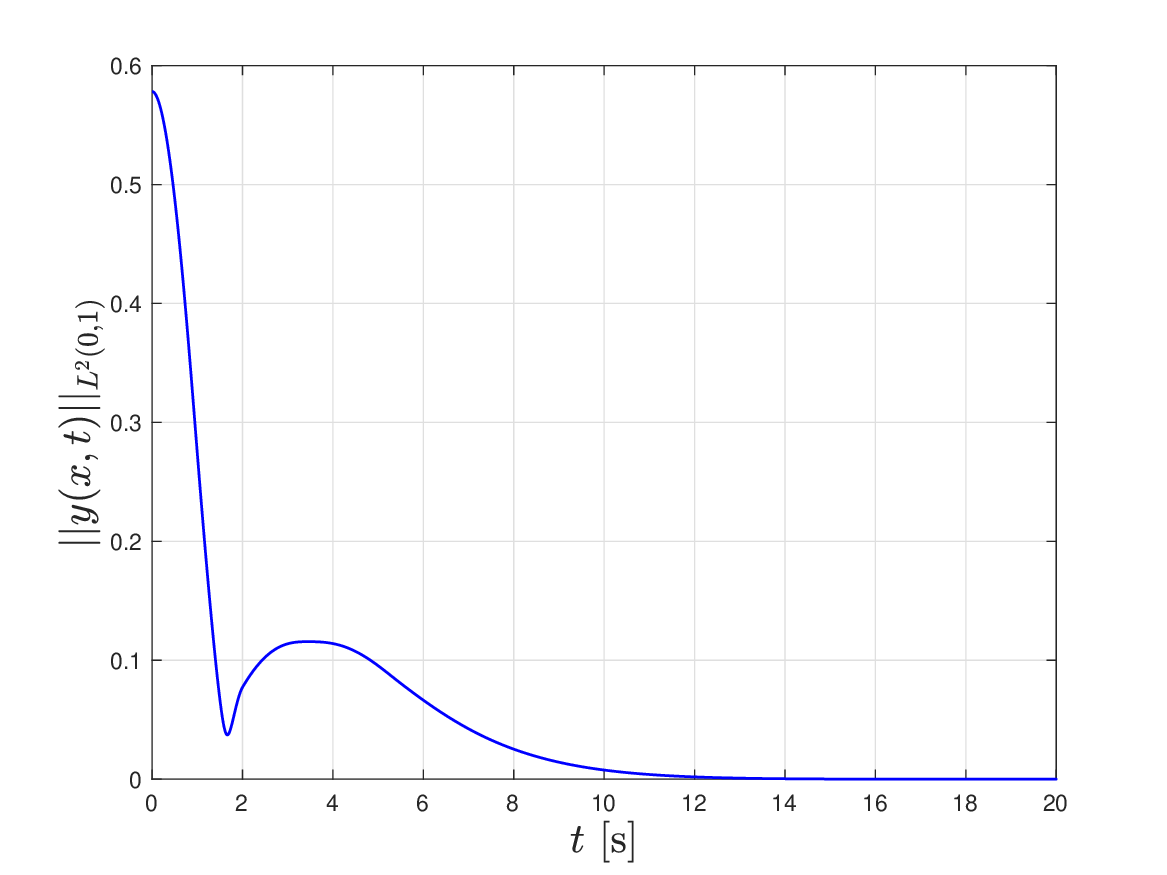}
     \caption{$\|y(x,t) \|_{L^2(0,1)}$.}\label{ynorm}
  \end{subfigure}
    \begin{subfigure}{4cm}
      \includegraphics[width=4cm]{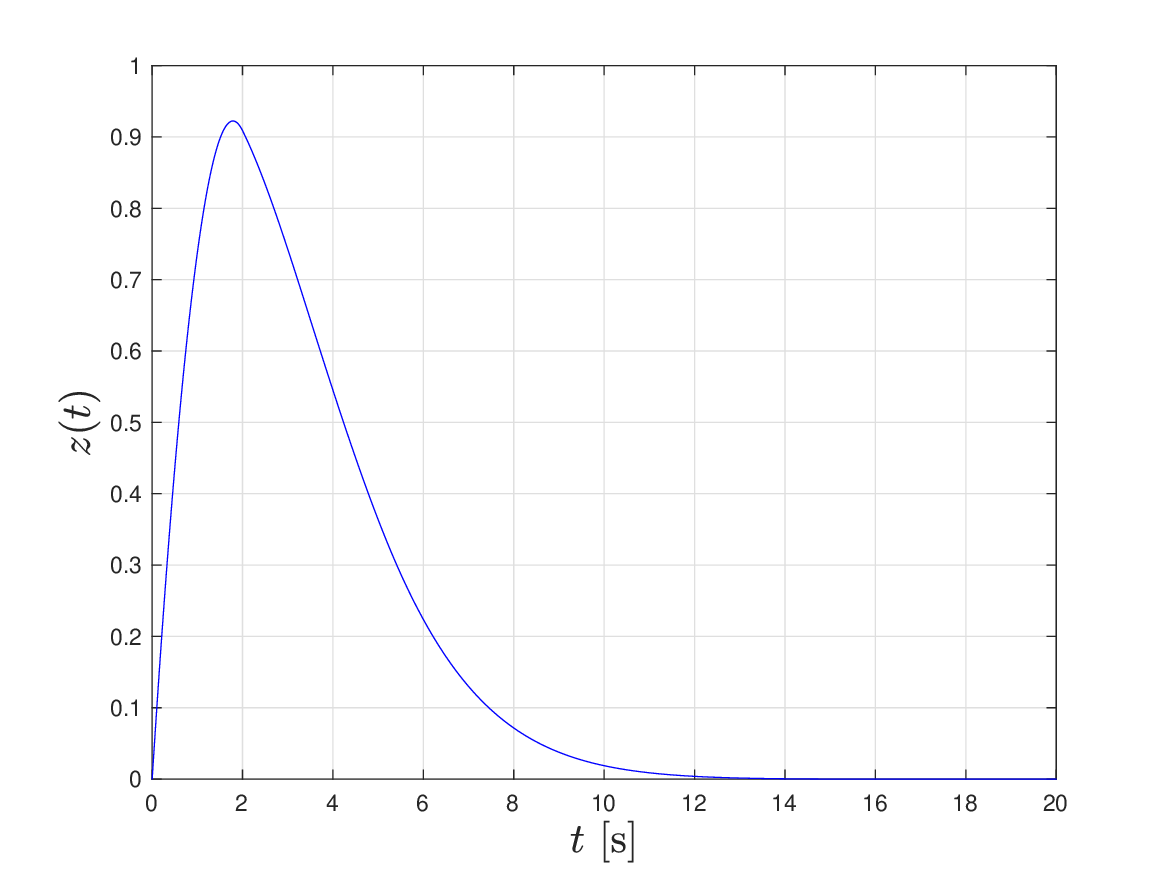}
      \caption{Evolution of $z(t)$.}\label{zt}
    \end{subfigure}
     \begin{subfigure}{4cm}
       \includegraphics[width=4cm]{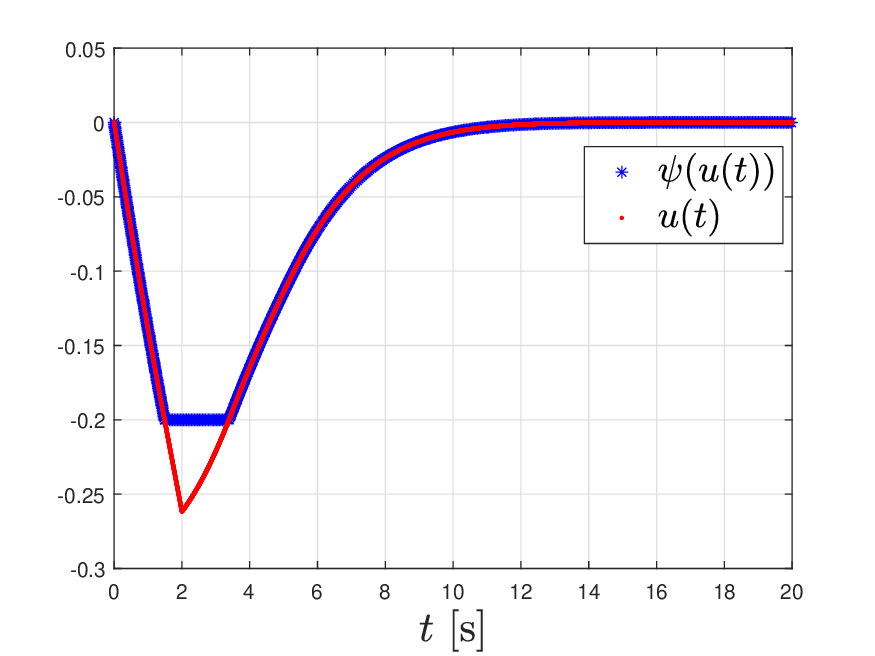}
       \caption{$u(t)$ (red) and $\psi(u(t))$ (blue).}\label{ut}
     \end{subfigure}
  \caption{The transient dynamics of wave equation \eqref{waveequation}-\eqref{eq_IntAction} under the controller \eqref{eq_ControlWave}.}
\end{figure}

To assess the effectiveness of the proposed control law \eqref{eq_ControlWave}, we present the simulation results below. We consider the initial condition $y_0(x)=x, y_{t0}(x)=0$ and $z_0 =0$ in \eqref{waveequation} and \eqref{eq_IntAction}. The controller parameter $\mu=0.3$ in \eqref{eq_ControlWave}. The nonlinear function $\psi$ is a saturation function of level $L =0.2$. 
The numerical simulations are conducted using a finite difference method with a step size of $0.002$ for both space and time.
 
Fig. \ref{yxt} and \ref{ynorm} depict the temporal evolution of the solution $y(x, t)$. Fig. \ref{ynorm} specifically represents the norm of $y(x, t)$. Moreover, Fig. \ref{zt} shows the convergence of $z(t)$ to zero. These visualizations provide evidence of the solution's convergence to its trivial equilibrium state. It emphasizes the asymptotic stability assessed in Theorem \ref{specwave}, i.e., an extension of Theorem \ref{theoremofsta} in the case of the wave equation with boundary control. Lastly, Fig. \ref{ut} illustrates the evolution of $u(t)$ and $\psi(u(t))$ under the influence of the saturation with $L = 0.2$. Despite activation of the saturation between about $1.5$~s and $3.4$~s, stabilization is achieved.

We conduct a series of simulations, as illustrated in Fig. \ref{fig:comparison_mu}, to  evaluate the impact of different controller parameter $\mu$,  The results presented in Fig. \ref{ztmu} reveal that higher values of $\mu$ induce apparent undershooting in the state $z(t)$. Moreover, in Fig. \ref{ytnormmu}, oscillations in the norm of the state $y(x,t)$ are observed with increasing $\mu$, particularly noticeable when $\mu=0.6$ in the time interval about $t=7-12$ s, exhibiting more pronounced vibrational behavior compared to other scenarios. Examining Fig. \ref{utmu}, it is evident that an increase in $\mu$ is associated with a rise in the control force $u(t)$ as defined in \eqref{eq_ControlWave} at $t=2$ s. Additionally, as depicted in Fig. \ref{satutmu}, larger values of $\mu$ result in greater values of $\psi(u(t))$. These results emphasize the diverse effects of $\mu$ on both state responses and control behavior. 

\begin{figure}
  \centering
  \begin{subfigure}{4cm}
    \includegraphics[width=4cm]{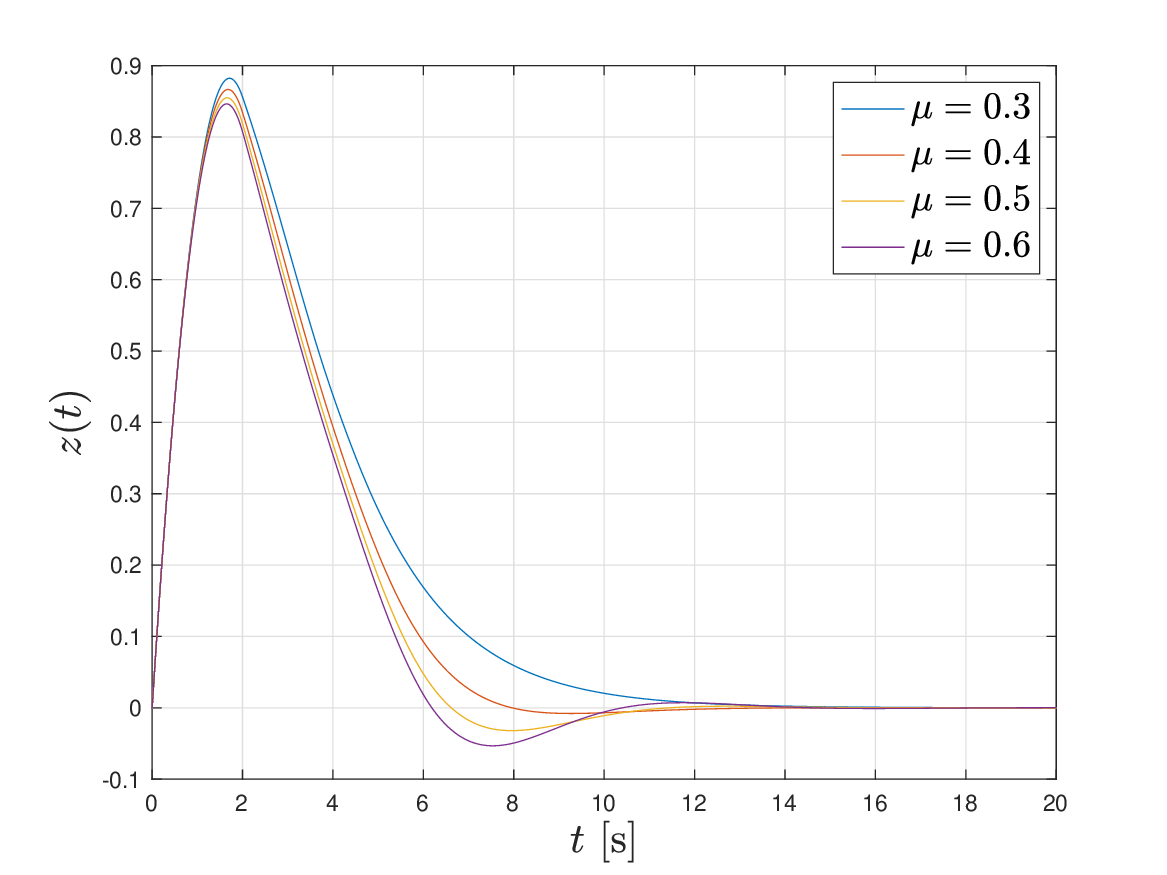}
    \caption{Evolution of $z(t)$.}\label{ztmu}
  \end{subfigure}
   \begin{subfigure}{4cm}
     \includegraphics[width=4cm]{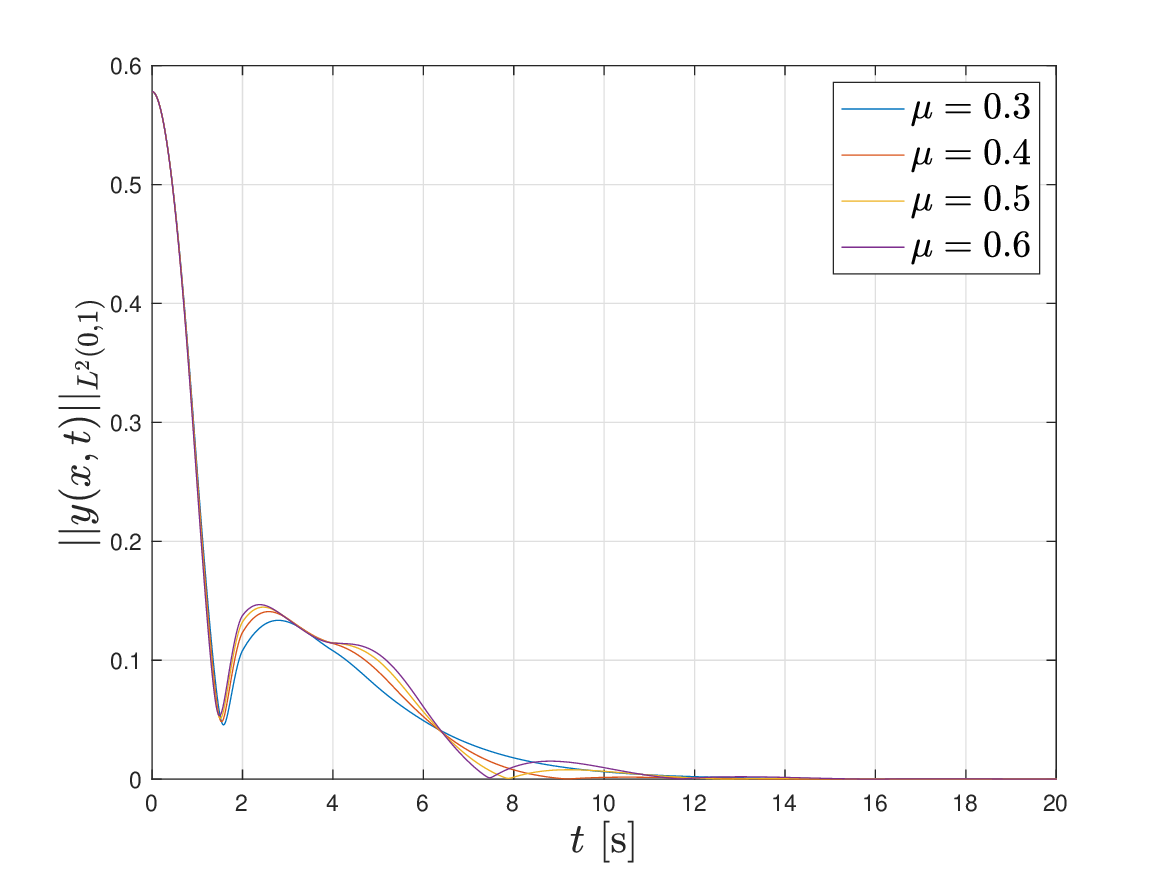}
     \caption{$\|y(x,t) \|_{L^2(0,1)}$.}\label{ytnormmu}
  \end{subfigure}
    \begin{subfigure}{4cm}
      \includegraphics[width=4cm]{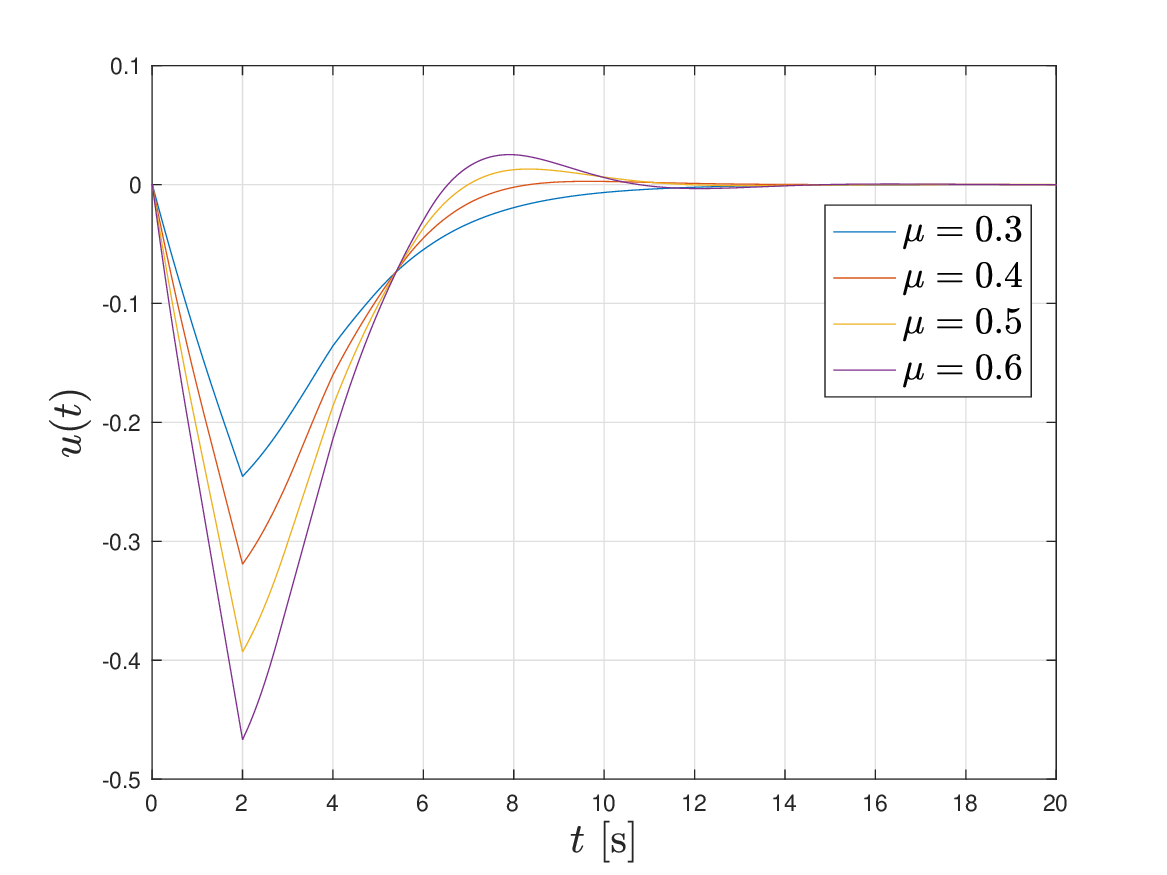}
      \caption{Evolution of $u(t)$.}\label{utmu}
    \end{subfigure}
     \begin{subfigure}{4cm}
       \includegraphics[width=4cm]{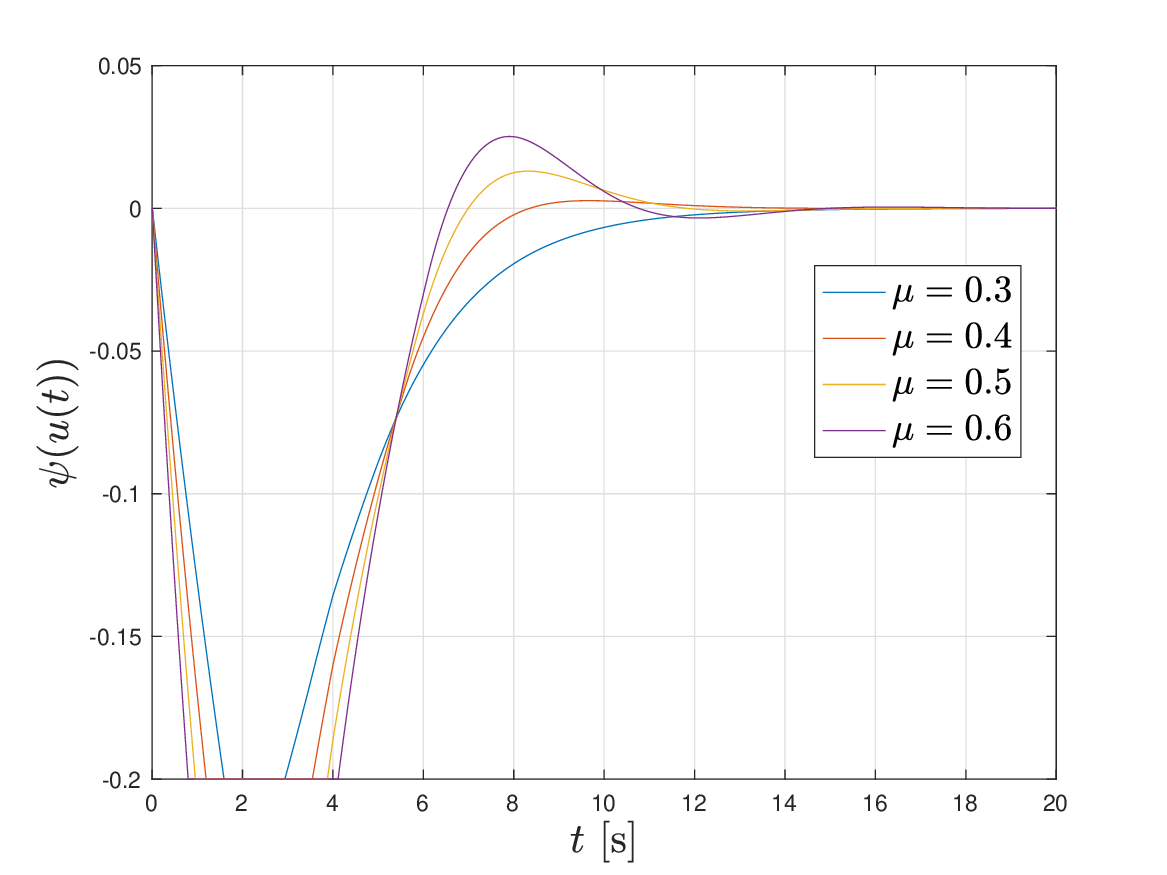}
       \caption{Evolution of $\psi(u(t))$.}\label{satutmu}
     \end{subfigure}
  \caption{Comparison results with different $\mu$.}
  \label{fig:comparison_mu}
\end{figure}

\section{Conclusion}
In this article, we demonstrated how it is possible to design an integral action for a partial differential equation system subject to nonlinearity of the actuator. This approach is made feasible through the use of weak Lyapunov functionals and forwarding-based approaches, revealing an approximate observability equation. The major challenge lies in the fact that, in this case, the distributed parameter system is not assumed to be exponentially stable in open-loop. This implies that a simple proportional control cannot solve the problem. The use of a more complex control, albeit at the cost of assuming the invertibility of the infinite-dimensional operator, allows us to address the issue.

As  future perspectives, it would be interesting to study 
new stabilization feedback in the presence of 
of multiple 
integral actions and input nonlinearities following 
the so-called nested-saturation approach
\cite{teel1992global,sussmann1993general}, and
possibly for unstable abstract systems. 
In this second case, similarly to the finite-dimensional case \cite{tarbouriech2011stability}, 
we believe that 
regional results and a detailed characterization of the domain of attraction should be carefully studied.


\bibliographystyle{IEEEtran}
\bibliography{bibsm}

\begin{thebibliography}{10}
\providecommand{\url}[1]{#1}
\csname url@samestyle\endcsname
\providecommand{\newblock}{\relax}
\providecommand{\bibinfo}[2]{#2}
\providecommand{\BIBentrySTDinterwordspacing}{\spaceskip=0pt\relax}
\providecommand{\BIBentryALTinterwordstretchfactor}{4}
\providecommand{\BIBentryALTinterwordspacing}{\spaceskip=\fontdimen2\font plus
\BIBentryALTinterwordstretchfactor\fontdimen3\font minus
  \fontdimen4\font\relax}
\providecommand{\BIBforeignlanguage}[2]{{%
\expandafter\ifx\csname l@#1\endcsname\relax
\typeout{** WARNING: IEEEtran.bst: No hyphenation pattern has been}%
\typeout{** loaded for the language `#1'. Using the pattern for}%
\typeout{** the default language instead.}%
\else
\language=\csname l@#1\endcsname
\fi
#2}}
\providecommand{\BIBdecl}{\relax}
\BIBdecl

\bibitem{prieur2016wave}
C.~Prieur, S.~Tarbouriech, and J.~M.~G. Da~Silva, ``Wave equation with
  cone-bounded control laws,'' \emph{IEEE Transactions on Automatic Control},
  vol.~61, no.~11, pp. 3452--3463, 2016.

\bibitem{marx2017cone}
S.~Marx, V.~Andrieu, and C.~Prieur, ``Cone-bounded feedback laws for
  $m$-dissipative operators on {H}ilbert spaces,'' \emph{Mathematics of
  Control, Signals, and Systems}, vol.~29, pp. 1--32, 2017.

\bibitem{jacob2020remarks}
B.~Jacob, F.~L. Schwenninger, and L.~A. Vorberg, ``Remarks on input-to-state
  stability of collocated systems with saturated feedback,'' \emph{Mathematics
  of Control, Signals, and Systems}, vol.~32, pp. 293--307, 2020.

\bibitem{chitour2021one}
Y.~Chitour, S.~Marx, and G.~Mazanti, ``One-dimensional wave equation with
  set-valued boundary damping: well-posedness, asymptotic stability, and decay
  rates,'' \emph{ESAIM: Control, Optimisation and Calculus of Variations},
  vol.~27, p.~84, 2021.

\bibitem{marx2021forwarding}
S.~Marx, L.~Brivadis, and D.~Astolfi, ``Forwarding techniques for the global
  stabilization of dissipative infinite-dimensional systems coupled with an
  {ODE},'' \emph{Mathematics of Control, Signals, and Systems}, vol.~33, pp.
  755--774, 2021.

\bibitem{pohjolainen1982robust}
S.~Pohjolainen, ``Robust multivariable {PI}-controller for infinite dimensional
  systems,'' \emph{IEEE Transactions on Automatic Control}, vol.~27, no.~1, pp.
  17--30, 1982.

\bibitem{xu1995robust}
C.-Z. Xu and H.~Jerbi, ``A robust {PI}-controller for infinite-dimensional
  systems,'' \emph{International Journal of Control}, vol.~61, no.~1, pp.
  33--45, 1995.

\bibitem{dos2008boundary}
V.~Dos~Santos, G.~Bastin, J.-M. Coron, and B.~d’Andr{\'e}a Novel, ``Boundary
  control with integral action for hyperbolic systems of conservation laws:
  Stability and experiments,'' \emph{Automatica}, vol.~44, no.~5, pp.
  1310--1318, 2008.

\bibitem{paunonen2015controller}
L.~Paunonen, ``Controller design for robust output regulation of regular linear
  systems,'' \emph{IEEE Transactions on Automatic Control}, vol.~61, no.~10,
  pp. 2974--2986, 2015.

\bibitem{Terrand2020}
A.~Terrand-Jeanne, V.~Andrieu, V.~D.~S. Martins, and C.-Z. Xu, ``Adding
  integral action for open-loop exponentially stable semigroups and application
  to boundary control of {PDE} systems,'' \emph{IEEE Transactions on Automatic
  Control}, vol.~65, no.~11, pp. 4481--4492, 2020.

\bibitem{coron2019pi}
J.-M. Coron and A.~Hayat, ``{PI} controllers for 1-{D} nonlinear transport
  equation,'' \emph{IEEE Transactions on Automatic Control}, vol.~64, no.~11,
  pp. 4570--4582, 2019.

\bibitem{vanspranghe2023output}
N.~Vanspranghe and L.~Brivadis, ``Output regulation of infinite-dimensional
  nonlinear systems: a forwarding approach for contraction semigroups,''
  \emph{SIAM Journal on Control and Optimization}, vol.~61, no.~4, pp.
  2571--2594, 2023.

\bibitem{balogoun2023iss}
I.~Balogoun, S.~Marx, and D.~Astolfi, ``{ISS} {L}yapunov strictification via
  observer design and integral action control for a {K}orteweg-de {V}ries
  equation,'' \emph{SIAM Journal on Control and Optimization}, vol.~61, no.~2,
  pp. 872--903, 2023.

\bibitem{astolfi2022global}
D.~Astolfi, S.~Marx, V.~Andrieu, and C.~Prieur, ``Global exponential set-point
  regulation for linear operator semigroups with input saturation,'' in
  \emph{IEEE 61st Conference on Decision and Control}, 2022, pp. 7358--7363.

\bibitem{lorenzetti2023saturating}
P.~Lorenzetti, L.~Paunonen, N.~Vanspranghe, and G.~Weiss, ``Saturating integral
  control for infinite-dimensional linear systems,'' in \emph{IEEE 62nd
  Conference on Decision and Control}, 2023.

\bibitem{logemann1998integral}
H.~Logemann, E.~Ryan, and S.~Townley, ``Integral control of
  infinite-dimensional linear systems subject to input saturation,'' \emph{SIAM
  journal on control and optimization}, vol.~36, no.~6, pp. 1940--1961, 1998.

\bibitem{logemann1999integral}
H.~Logemann, E.~P. Ryan, and S.~Townley, ``Integral control of linear systems
  with actuator nonlinearities: lower bounds for the maximal regulating gain,''
  \emph{IEEE transactions on automatic control}, vol.~44, no.~6, pp.
  1315--1319, 1999.

\bibitem{terrand2019regulation}
A.~Terrand-Jeanne, V.~Andrieu, M.~Tayakout-Fayolle, and V.~D.~S. Martins,
  ``Regulation of inhomogeneous drilling model with a {PI} controller,''
  \emph{IEEE Transactions on Automatic Control}, vol.~65, no.~1, pp. 58--71,
  2019.

\bibitem{slemrod1989feedback}
M.~Slemrod, ``Feedback stabilization of a linear control system in {H}ilbert
  space with an a priori bounded control,'' \emph{Mathematics of Control,
  Signals and Systems}, vol.~2, pp. 265--285, 1989.

\bibitem{Isao1992}
I.~Miyadera, \emph{Nonlinear Semigroups}.\hskip 1em plus 0.5em minus
  0.4em\relax American Mathematical Society, 1992.

\bibitem{brezis1973ope}
H.~Brezis, \emph{Operateurs maximaux monotones et semi-groupes de contractions
  dans les espaces de {H}ilbert}.\hskip 1em plus 0.5em minus 0.4em\relax
  Elsevier, 1973.

\bibitem{Coron2007}
J.-M. Coron, \emph{Control and Nonlinearity}.\hskip 1em plus 0.5em minus
  0.4em\relax American Mathematical Society, 2007.

\bibitem{tucsnak2014well}
M.~Tucsnak and G.~Weiss, ``{Well-posed systems—the LTI case and beyond},''
  \emph{Automatica}, vol.~50, no.~7, pp. 1757--1779, 2014.

\bibitem{teel1992global}
A.~R. Teel, ``Global stabilization and restricted tracking for multiple
  integrators with bounded controls,'' \emph{Systems \& control letters},
  vol.~18, no.~3, pp. 165--171, 1992.

\bibitem{sussmann1993general}
H.~Sussmann, E.~Sontag, and Y.~Yang, ``A general result on the stabilization of
  linear systems using bounded controls,'' in \emph{32nd IEEE Conference on
  Decision and Control}, 1993, pp. 1802--1807.

\bibitem{tarbouriech2011stability}
S.~Tarbouriech, G.~Garcia, J.~M.~G. da~Silva~Jr, and I.~Queinnec,
  \emph{Stability and Stabilization of Linear Systems with Saturating
  Actuators}.\hskip 1em plus 0.5em minus 0.4em\relax Springer Science \&
  Business Media, 2011.

\end{thebibliography}

\end{document}